\documentclass[11pt,reqno]{amsart}

\usepackage[utf8]{inputenc}
\usepackage[margin=1.25in]{geometry}
\usepackage{hyperref}
\usepackage{appendix}
\usepackage{amsfonts}
\usepackage{amssymb}
\usepackage{stmaryrd} 
\usepackage{amsmath}
\usepackage{amsthm}
\usepackage{mathrsfs}

\theoremstyle{plain}
\newtheorem{theorem}{Theorem}[section]

\newtheorem{lemma}[theorem]{Lemma}
\newtheorem{proposition}[theorem]{Proposition}
\theoremstyle{remark}

\newtheorem{Definition}[theorem]{Definition}
\newtheorem{Assumption}[theorem]{Assumption}
\newtheorem{remark}[theorem]{Remark}

\numberwithin{equation}{section}

\usepackage{biblatex} 
\title[ Propagation of chaos with $L^p$ interactions]{Entropic propagation of chaos for mean field diffusion with $L^p$ interactions via hierarchy, linear growth and fractional noise}
\author{Yi Han}
\address{Department of Pure Mathematics and Mathematical Statistics, University of Cambridge.}
\email{yh482@cam.ac.uk}
\thanks{Supported by EPSRC grant EP/W524141/1.}

\begin{document}

\maketitle

\begin{abstract} 
New quantitative propagation of chaos results for mean field diffusions are proved via local and global entropy estimates. In the first result we work on the torus and consider singular, divergence free interactions $K\in L^p$, $p>d$. We prove an $O(k^{2}/n^2)$ convergence rate in relative entropy between the $k$-th marginal laws of the particle system and its limiting law at each time $t$, as long as the same holds at time 0. The proof is based on local estimates via a form of BBGKY hierarchy and exemplifies a method to extend the framework in Lacker \cite{lacker2021hierarchies} to singular interactions. The rate can be made uniform in time combined with the result in \cite{lacker2022sharp}.

Then we prove quantitative propagation of chaos for interactions that are only assumed to have linear growth. This generalizes to the case where the driving noise is replaced by a fractional Brownian motion $B^H$, for all $H\in(0,1)$. These proofs follow from global estimates and subGaussian concentration inequalities. We obtain $O(k/n)$ convergence rate in relative entropy in each case, yet the rate is only valid on $[0,T^*]$ with $T^*$ a fixed finite constant depending on various parameters of the system.
\end{abstract}

\section{Introduction}

In this paper, we study the convergence of the interacting particle system
\begin{equation}\label{nparticlesystem}d X_{t}^{ i}=\left(\frac{1}{n-1} \sum_{j \neq i} b\left(t,X_{t}^{i}, X_{t}^{ j}\right)\right) d t+d W_{t}^{i},\quad i=1\cdots n,\end{equation}
in the $n\to\infty$ limit, to the corresponding McKean-Vlasov SDE
\begin{equation}\label{mckeanvlasov}
d X_{t}=\left(\langle\mu_t, b(t, X_t, \cdot)\rangle\right) d t+d W_t, \quad \mu_t=\operatorname{Law}(X_t),
\end{equation} 
where $W^1,\cdots,W^n$ are $n$ independent $d$-dimensional Brownian motions. 

There has been vast literature on the quantitative convergence of  \eqref{nparticlesystem} towards \eqref{mckeanvlasov}, even in the general case of irregular interactions $b$. In this paper we address the following three questions:
\begin{enumerate}
    \item 
 In a very recent paper, Lacker \cite{lacker2021hierarchies} developed a form of BBGKY hierarchy, from which he proved that the $k$-th marginals of \eqref{nparticlesystem} converge to the $k$-fold product of \eqref{mckeanvlasov} with a rate $O(k^2/n^2)$ measured in terms of relative entropy, vastly improving the $O(k/n)$ rate which was previously assumed to be optimal. The technique in \cite{lacker2021hierarchies} is restricted to interactions $b$ that are locally bounded, which falls short of many interesting physical cases like the Biot-Savart kernel in 2D incompressible Navier-Stokes equation. We extend his framework to some singular interactions on the torus and get the same $O(k^2/n^2)$ rate in relative entropy. Moreover, the rate can be made uniform in time. (Another recent work \cite{bresch2022new} used PDE techniques and weighted $L^p$ spaces to set up the BBGKY hierarchy, see Section \ref{comparisons} for a comparison.)

\item For the remaining results we work on $\mathbb{R}^d$. We consider interactions $b$ only assumed to have linear growth: $|b(t,x,y)|\leq K(1+|x|+|y|)$ for some $K>0$. We can still prove a quantitative propagation of chaos result with rate $O(k/n)$ in relative entropy under such generality, but only valid on $[0,T^*]$ with $T^*$ a predetermined constant.

\item We replace the driving Brownian motion $W$ by a fractional Brownian motion $B^H$.  Quantitative propagation of chaos results (on short time $[0,T^*]$) are proved with the same $O(k/n)$ rate for all $H\in(0,1)$, and the (locally bounded) interactions only need to have minimal regularity such that the particle systems are well-posed.
\end{enumerate}

The following table lists some representative literature in the study of propagation of chaos for interacting diffusion. It serves as a road map and indicates the position of our findings among the literature. The $O(k^2/n^2)$ and $O(k/n)$ rates are in terms of relative entropy, which implies, by Pinsker's inequality, the $O(k/n)$ and $O((k/n)^{1/2})$ rate in total variation. The enumeration is by no means exhaustive, and we refer to \cite{chaintron2021propagation} for a rather complete literature review. 
\begin{center}
\begin{tabular}{|c|c|}%
\hline  
Uniform in time propagation of chaos &\cite{malrieu2003convergence}, \cite{lacker2022sharp},\cite{guillin2021uniform}, \cite{rosenzweig2021global}, Section \ref{uniformsec}.\\
\hline  
$O(k^2/n^2)$ rate on any finite interval $[0,T]$ & Condition \eqref{1.2} in \cite{lacker2021hierarchies}; Theorem \ref{integrablccase}.\\
\hline 
$O(k/n)$ rate on any finite interval $[0,T]$& Lipschitz case:\cite{10.1007/BFb0085169}; singular case:\cite{jabin2018quantitative}.\\
\hline
$O(k/n)$ rate on fixed short time $[0,T^*]$&Theorem \ref{theorem01},  \ref{theorem02}, \ref{theorem2}. \\
\hline 
Qualitative convergence without a rate & See for example \cite{tomasevic:hal-03086253}, \cite{hao2022strong}, \cite{lacker2018strong}.
\\
\hline
Well-posedness of the limiting equation &All the above. See for example \cite{han2022solving}.\\ 
\hline
\end{tabular}
\end{center}

We briefly explain the philosophy behind this table. First, irregularity of the interactions $b$ can be tackled with well behaved initial conditions, global stability property of the system, and compactness of domains. This is the heart of Theorem \ref{integrablccase}, a generalization of \cite{lacker2018strong} to singular interactions. Second, quantitative propagation of chaos on any finite interval $[0,T]$ is notably harder than that on short time $[0,T^*]$, as some martingale structure could be necessary. If one only considers sufficiently short time $[0,T^*]$, then very little information on the system is enough: some Gaussian concentration phenomenon is sufficient. This allows us to consider the minimal assumption on interaction $b$ (linear growth condition, Theorem \ref{theorem01}), and replace the Brownian motion by its fractional counterpart, see Theorem \ref{theorem02}, \ref{theorem2}. We refer to footnote \ref{footnote123} for relevant issues.

We fix some notations throughout the article.
For a polish space $E$ denote by $\mathcal{P}(E)$ the space of probability measures on $E$.

The $d$-dimensional 1-periodic torus $\mathbb{T}^d$ is identified with $[-\frac{1}{2},\frac{1}{2}]^d$ with periodic boundary condition.

For each $t\geq 0$, denote by
$\mathcal{C}_{t}^{d}:=C\left([0, t] ; \mathbb{R}^{d}\right)$  the space of continuous, $\mathbb{R}^d$-valued paths with the supremum norm
$$
\|x\|_{t}:=\sup _{s \in[0, t]}\left|x_{s}\right|.
$$
For a probability measure $\mu\in \mathcal{P}(\mathcal{C}_T^d)$ and any $0\leq t\leq T$, denote by $\mu[t]$ the projection of $\mu$ onto $\mathcal{C}_t^d$, via the restriction $x\mapsto x|_{[0,t]}$. Also denote by $\mu_t$ the time-$t$ marginal of $\mu$.

$H\left(\cdot\mid\cdot\right)$ denotes the relative entropy on any measurable space, defined as
\begin{equation}\label{relativeentropy}H\left(\nu\mid\mu\right):=\int\frac{d\nu}{d\mu}\log \frac{d\nu}{d\mu}d\mu,\end{equation}
with the convention that $H\left(\nu\mid\mu\right)=\infty$ if $\nu$ is not absolutely continuous with respect to $\mu$.

\subsection{BBGKY hierarchy for singular interactions} This section exemplifies a generalization of the framework in Lacker \cite{lacker2021hierarchies} to singular interactions. The model we consider is inspired by \cite{guillin2021uniform}.

For any smooth manifold $\chi$ denote by $\mathcal{C}^\infty(\chi)$ the space of real valued smooth functions on $\chi$, and for any $\lambda>1$ denote by $\mathcal{C}_\lambda^\infty(\chi)$ the functions $f\in \mathcal{C}^\infty(\chi)$ such that $0<\frac{1}{\lambda}\leq f\leq\lambda<\infty$.

\begin{Assumption}\label{assumption1}
We consider $K:\mathbb{T}^d\to\mathbb{R}^d$ satisfying the following conditions:
\begin{itemize}
    \item  $\|K\|_{L^p}<\infty$ for some $p>d$,
    \item $\nabla\cdot K=0$ in the sense of distributions.
    \end{itemize}
\end{Assumption}

Under assumption \ref{assumption1}, strong well-posedness of the interacting particle system \eqref{thm1nparticlesystem} can be found in \cite{hao2022strong}. That of the McKean-Vlasov SDE \eqref{thm1mckeanvlasov} can be found in \cite{rockner2021well} or \cite{han2022solving}

\begin{theorem}\label{integrablccase}
Fix $K$ satisfying Assumption \ref{assumption1}.
Consider the interacting particle system \begin{equation}\label{thm1nparticlesystem}d X_{t}^{ i}=\left(\frac{1}{n-1} \sum_{j \neq i} K\left(X_{t}^{ i}-X_{t}^{ j}\right)\right) d t+d W_{t}^{i},\quad i=1\cdots n,\end{equation} and the McKean-Vlasov SDE \begin{equation}\label{thm1mckeanvlasov}
d X_{t}=\left(\langle\mu_t, K( X_t- \cdot)\rangle\right) d t+d W_t, \quad \mu_t=\operatorname{Law}(X_t).
\end{equation} 
Assume the following conditions are satisfied: for some $\lambda>1$,
\begin{enumerate}
    \item The initial law $P_0^{(n,n)}$ of \ref{thm1nparticlesystem} is exchangeable. 
    Denote by $P_0^{(n,k)}$ its $k$-th marginal.
    \item \label{condition2} $P_0^{(n,n)}$ has a density $\rho_n(0)$ such that $\rho_n(0,\cdot)\in\mathcal{C}_\lambda^\infty((\mathbb{T}^d)^n)$, uniform in $n$.
    \item \label{condition3} The initial law $\mu_0$ of \eqref{thm1mckeanvlasov} has a density $\bar{\rho}_0\in\mathcal{C}_\lambda^\infty(\mathbb{T}^d)$. 
    
\item \label{condition4}The chaotic initial condition: there exists some $C_0>0$ such that  $$H\left(P_0^{(n,k)}\mid \mu_0^{\otimes k}\right)\leqslant C_0\frac{k^2}{n^2},\quad k=1,\cdots,n.$$
\end{enumerate}

 Then by weak uniqueness, the joint law $P_t^{(n,n)}$ of $(X_t^{1},\cdots,X_t^{n})$ is exchangeable. Denote by $P_t^{(n,k)}$ its $k$-th marginal, then we have the following propagation of chaos result: 
 
 There exists constants $\gamma$, $M$ and $C$ depending on $\lambda$, $K$ and $T$ such that, if $n\geq 6e^{\gamma T}$, then for all $t\in[0,T]$ and $k=1,\cdots,n$,
\begin{equation}
    H\left(P_t^{(n,k)}\mid \mu_t^{\otimes k}\right)\leq 2C\frac{k^2}{n^2}+C\exp(-2n(e^{-\gamma T}-\frac{k}{n})_+^2).
\end{equation}

\end{theorem}
By Pinsker's inequality this implies a $O\left(\frac{1}{n}\right)$ rate in total variation:
$$\|\mu_t^{\otimes k}-P_t^{(n,k)}\|_{TV}=O\left(\frac{k}{n}\right),\quad n\to\infty,\frac{k}{n}\to 0,$$
where $\|\cdot\|_{TV}$ denotes the total variation distance.

This theorem is proved in Section \ref{section22}.

\subsubsection{Uniform in time estimates}\label{uniformsec}

By the assumptions in Theorem \ref{integrablccase}, the density function $\bar{\rho}(t,\cdot)$ of the McKean-Vlasov SDE \eqref{thm1mckeanvlasov} satisfies $\frac{1}{\lambda}\leq \bar{\rho}(t,\cdot)\leq\lambda$ for all $t\geq 0$, see Proposition \ref{esttrans1}. This implies a Logarithmic Sobolev inequality for $\bar{\rho}(t,\cdot)$ that is uniform in $t\in(0,\infty)$: there exists some constant $C>0$ such that for all $f\in \mathcal{C}_{>0}^\infty(\mathbb{T}^d)$, 
$$\operatorname{Ent}_{\bar{\rho}(t,\cdot)}(f)\leq C \int_{\mathbb{T}^d}\frac{|\nabla f|^2}{f}d\bar{\rho}(t,\cdot).$$

Moreover, the Log-Sobolev inequality tensorizes, with the same constant $C$ for any $\bar{\rho}^{\otimes k}(t,\cdot)$ and $k\geq 1$. The proof of these facts can be found in Lemma 3 and Theorem 3 of \cite{guillin2021uniform}. Via the Log-Sobolev inequality, Equation \eqref{differentialinequality12}, which is the main iterative inequality we want to solve, boils down to 
\begin{equation}
    \frac{d}{dt}H_t^k\leq -CH_t^k+ \frac{k(k-1)^2}{(n-1)^2}+\gamma k(H_t^{k+1}-H_t^k).
\end{equation}

Now we can invoke the very recent computational results in Theorem 2.1 of \cite{lacker2022sharp} and obtain a $O(k^2/n^2)$ rate that is uniform in $t\in(0,\infty)$, under an extra smallness condition of the interaction $K$. Note however that this smallness condition is not required in \cite{guillin2021uniform}, which gives the $O(k/n)$ rate.

\subsubsection{Applications to the Biot-Savart kernel}
The Biot-Savart kernel on $\mathbb{R}^2$ is given by $$\widetilde{K}(x)=\frac{1}{2\pi}\frac{x^\perp}{|x|^2},$$
where for $x=(x_1,x_2)$, $  x^\perp=(x_2,-x_1)$. The Biot-Savart kernel appears in the 2D incompressible fluid dynamics and is one of the central models studied in the propagation of chaos literature ( \cite{jabin2018quantitative} and \cite{guillin2021uniform}). On the torus $\mathbb{T}^2$, $\widetilde{K}$ admits a version (see equation (1.6) of \cite{guillin2021uniform})
\begin{equation}\label{biotsavart}
K(x)=\frac{1}{2\pi}\frac{x^\perp}{|x|^2}+\frac{1}{2\pi}\sum_{k\in\mathbb{Z}^2,k\neq 0}\frac{(x-k)^\perp}{|x-k|^2},\end{equation}
that satisfies $\nabla\cdot K=0$ in the sense of distribution. The interaction $K$ is not covered by Assumption \ref{assumption1} in this article, yet it is covered by the following more general assumption (Assumption 2 in \cite{guillin2021uniform}), the last condition of which is justified by Proposition 2 of \cite{jabin2018quantitative}:

\begin{Assumption}\label{assumption2} We impose the following more general condition on $K$:
\begin{itemize}
    \item $\|K\|_{L^1}<\infty.$
    \item $\nabla\cdot K=0$ in the sense of distributions.
    \item There exists $V\in L^\infty$ such that $K=\nabla\cdot V$.
\end{itemize}
\end{Assumption}

Elementary computation shows the Biot-Savart kernel $K$ given in \eqref{biotsavart} satisfies  $K\in L^p(\mathbb{T}^2)$ for $p\in[1,2)$ but $K\notin L^2(\mathbb{T}^2)$. Taking the Biot-Savart kernel $K$ into the proof of Theorem \ref{integrablccase}, the proof remains formally correct despite two technical issues:
\begin{enumerate}
    \item The well-posedness of \eqref{thm1nparticlesystem} as a stochastic particle system is unknown. We may replace the SDEs by the corresponding Fokker-Planck PDEs. The well-posedness of the so-called \textit{entropy solutions} of Fokker-Planck PDEs for the Biot-Savart kernel $K$ are proven in \cite{jabin2018quantitative}. 
    \item In the proof we need to establish inequality \eqref{integrablefirst}, which follows if $\|K\|_{L^2(\mathbb{T}^d)}<\infty$. This is not verified by the Biot-Savart kernel $K$. Currently, our proof strategy does not seem to be applicable in the case when $\|K\|_{L^2(\mathbb{T}^d)}=\infty$. 
\end{enumerate}

To simplify presentations, we will restrict to the subcritical case, i.e., Assumption \ref{assumption1}. 

\subsubsection{On the initial condition}
The assumption (2) in Theorem \ref{integrablccase} is notably stronger than its counterpart in \cite{guillin2021uniform}. In particular, $\mu_0$ has to be the probability measure that is identically 1 on $\mathbb{T}^d$. Yet to our best knowledge, this is the first work that addresses the $O((k^2/n^2))$ convergence rate for a singular interaction kernel, and which gives uniform in time estimates. We are not sure if assumption (2) in Theorem \ref{integrablccase} is physically meaningful, or if it can be weakened in the future to the following form: for some $\lambda>1$,
\begin{equation}
    \frac{1}{\lambda^n}\leq \rho_n(0,\cdot)\leq \lambda^n,\quad \text{ for all } n\geq 1.
\end{equation}

\subsubsection{Comparison with recent work} \label{comparisons}
In a very recent work \cite{bresch2022new} a new approach to mean field limit via BBGKY hierarchy is introduced. The $L^p$ estimates of the joint density function $\rho_n(t,\cdot)$ are used in place of information-theoretic functional inequalities. The estimates obtained in that paper do not lead to a Grönwall argument, and thus (due to its proof technique) is valid only on short time $[0,T^*]$ with $T^*$ depending on various parameters of the system. In Theorem \ref{integrablccase}, we assumed a stronger initial condition \ref{condition2} and less singular interaction \ref{assumption1}, then obtained estimates on any finite time interval $[0,T]$ via the BBGKY hierarchy, and which can be uniform in time.

\subsubsection{Main idea of the proof}
We now explain our strategy in generalizing the BBGKY hierarchy in Lacker \cite{lacker2021hierarchies} to singular interactions. Generalization to the singular case is more delicate than it may seem. Lying in the core is the following transport-type inequality (Condition (3) of \cite{lacker2021hierarchies}, Theorem 2.2): for some $0<\gamma<\infty$,
\begin{equation}\label{maintransport}|\langle \nu-\mu,b(t,x,\cdot)\rangle|^2\leq \gamma H(\nu\mid\mu)\quad \text{ for all } t\in[0,T],x\in\mathcal{C}_T^d,\nu\in\mathcal{P}(\mathcal{C}_T^d),\end{equation} where $\mu\in\mathcal{P}(\mathcal{C}_T^d)$ is the law of the McKean-Vlasov SDE \eqref{mckeanvlasov}. When $b$ is bounded and measurable, \eqref{maintransport} follows from Pinsker's inequality. When $b$ is Lipschitz, \eqref{maintransport} follows from  Kantorovich duality (bounding the left hand side by Wasserstein distance $\mathcal{W}_1$) and $T_1$ transport-entropy inequality (bounding $\mathcal{W}_1$ by relative entropy $H$). For more general cases, the weighted Pinsker's inequality (recalled in \eqref{weightedpinsker}) can be used, which in the case $b(t,x,y)=K(x-y)$, essentially requires $\int e^{\lambda |K(x)|^2}d\mu(x)<\infty$ for some $\lambda>0$. Such integrability condition cannot be verified in the case $K\in L^p(\mathbb{R}^d)$ and  $\mu$ the Gaussian measure.

A closer look at the proof of the weighted Pinsker's inequality shows that the exponential integrability condition is a consequence of the Gibbs variational formula 
$$\int\varphi^2 d\mu\leq H(\mu\mid\nu)+\log\int e^{\varphi^2}d\nu$$ as we do the change of measure, which holds for any two probability measures $\mu$, $\nu$ and measurable function $\varphi$ on the same measurable space.

 The idea is, if we have a bit more information of $\mu$ and $\nu$, we can possibly avoid using the Gibbs variational formula as we do the change of measure, and consequently get rid of the requirement that $e^{\varphi^2}$ is integrable. A careful use of Cauchy-Schwartz and Jensen's inequality gives the following result: (see equation (3.8) and (3.9) of \cite{AFST_2005_6_14_3_331_0}):

\begin{equation}
\label{reducedpinsker}    
\left(\int \varphi d|\mu-\nu|\right)^2\leq 4CH\left(\mu\mid\nu\right), \quad C=\frac{1}{6}\int\varphi^2d\mu+\frac{1}{3}\int\varphi^2d\nu.\end{equation}

We stress that when applying the transport-type inequality \eqref{maintransport}, we only need \eqref{maintransport} to hold for all $\nu$ of the type $P_{X^1,\cdots,X^k}^{(k+1\mid k)}$ (a version of the regular conditional law of $X^{k+1}$ given $(X^1,\cdots,X^k)_{[0,T]}$) instead of ranging over all $\nu\in\mathcal{P}(\mathcal{C}_T^d)$.
Thus we will take $\mu=\mu[t]$ the law of \eqref{mckeanvlasov}; take $\nu:=P_{X^1,\cdots,X^k}^{(k+1\mid k)}$; and take $\varphi:=K(X_t^1-\cdot)$ in the case $b(t,x,y)=K(x-y)$.

The target is to find an upper bound of $C$ that is independent of $(X^1,\cdots,X^k)_{0\leq t\leq T}.$ In the case $K\in L^p(\mathbb{R}^d)$, this corresponds, via Hölder's inequality, to find a $L^{q}$-estimate, for some exponent $q$, of the probability density function $P_{X^1,\cdots,X^k}^{(k+1\mid k)}$. This estimate should be uniform in $X^1\cdots X^k$.

The conditional law $P_{X^1,\cdots,X^k}^{(k+1\mid k)}$ is, in general, hard to manipulate. This is because $X^1,\cdots,X^k$ are not independent from the Brownian motion that drives $X^{k+1}$, so the conditional law $P_{X^1,\cdots,X^k}^{(k+1\mid k)}$ cannot be represented as the solution of an SDE. \footnote{We thank Daniel Lacker for pointing out this issue.} A promising route to bound $P_{X^1,\cdots,X^k}^{(k+1\mid k)}$ that we can think of, is to find a pointwise upper bound and  nonzero lower bound of the joint density function of $(X^1,\cdots,X^n)$, such that a pointwise upper bound for $P_{X^1,\cdots,X^k}^{(k+1\mid k)}$ follows for free.

Obviously, such an operation is possible only if one works on the torus $\mathbb{T}^d$ rather than the Euclidean space $\mathbb{R}^d$, and only if we assume the initial law of \eqref{thm1nparticlesystem} has a density that is uniformly upper and lower bounded. Similar assumptions also appear in the seminal work of Jabin and Wang \cite{jabin2018quantitative}, where PDE techniques are used instead. We expect that our proof can give a probabilistic interpretation to the assumptions and techniques in \cite{jabin2018quantitative}. Unfortunately, our assumption (2) in the statement of Theorem \ref{integrablccase} is too restrictive and cannot fulfill the goal of giving such a satisfying probabilistic interpretation.

These discussions suggest that a global stability property is fundamental when applying the BBGKY hierarchy to mean field diffusion. We can expect a trade-off between (i) the regularity of the interaction, and (ii) the regularity of initial laws and the compactness of domain. For locally bounded interactions considered in \cite{lacker2021hierarchies}, we can work with the whole space $\mathbb{R}^d$ and the conditions on the initial law is rather loose: a second order exponential moment $\int_{\mathbb{R}^d}e^{c|x|^2}\mu_0(dx)<\infty$ (for some $c>0$) is sufficient. (This does not contradict the previous paragraph which discusses singular interactions, as here we discuss the locally bounded interaction case). The condition \eqref{1.2} is imposed so that we can verify \eqref{maintransport}. In order to apply the same BBGKY hierarchy for  singular interactions, we restrict to the torus and require conditions \ref{condition2} and \ref{condition3} in Theorem \ref{integrablccase} (it could be possible to remove one of them). We also assume $\nabla\cdot K=0$ in Assumption \ref{assumption1} so that the system is globally stable. These two conditions, though apparently disparate, are just the two faces of the same coin that guarantees \eqref{transporttype} is valid independent of $(x_1,\cdots,x_k)$.

These demanding conditions are not necessary when one only considers qualitative mean field limit, i.e. proving the weak convergence of $L_t^n:=\frac{1}{n}\sum_{i=1}^n\delta_{X_t^{n,i}}$ towards the limiting law $\mu_t$. Indeed, qualitative mean field limit has been established in \cite{hoeksema2020large}, \cite{tomasevic:hal-03086253}, \cite{tomavsevic2023propagation} for interactions satisfying Krylov's conditions $|b(t,x,y)|\leq h_t(x-y)$ with $h\in L^q([0,T];L^p(\mathbb{R}^d)$, for $\frac{d}{p}+\frac{2}{q}<1$. This condition is more general than Assumption \ref{assumption1} in this paper. See also \cite{jabir2018mean} for related works, where only qualitative convergence were proved. We note in passing that \cite{hoeksema2020large} proved this propagation of chaos result via a large deviations approach whereas \cite{tomasevic:hal-03086253} and \cite{tomavsevic2023propagation} adopts a completely different approach: the partial Girsanov transform.

\subsection{Interactions only assumed to have linear growth}
In the next two results we consider interactions that can be path dependent, i.e. the interaction is given by a progressively measurable function $b:[0,T]\times\mathcal{C}_T^d\times\mathcal{C}_T^d\to\mathbb{R}^d$. Progressive measurability is defined as follows (see \cite{lacker2021hierarchies}, Section 2.1.2): identifying $\mathcal{C}_T^d\times\mathcal{C}_T^d$ with $\mathcal{C}_T^{2d}$, then for $k \in \mathbb{N}$ we say that a function $\bar{b}:[0, T] \times \mathcal{C}_{T}^{k} \rightarrow \mathbb{R}$ is progressively measurable if (i) it is Borel measurable and (ii) it is non-anticipative, which means $\bar{b}(t, x)=\bar{b}\left(t, x^{\prime}\right)$ for every $t \in[0, T]$ and $x, x^{\prime} \in C_{T}^{k}$ satisfying $\left.x\right|_{[0, t]}=\left.x^{\prime}\right|_{[0, t]} .$

We return to the classical case of interactions $b$ that are locally bounded. A common assumption that ensures the particle system  \eqref{variant1} (the path-dependent variant of\eqref{nparticlesystem} with an extra drift) to be weakly well-posed is the following linear growth condition when the driving noise is Brownian motion (see for example \cite{karatzas1991brownian} and \cite{liptser2001statistics}): for some constant $0<K<\infty$,
\begin{equation}\label{linear}\left|b_{0}(t, x)\right|+|b(t, x, y)| \leq K\left(1+\|x\|_{t}+\|y\|_{t}\right)\quad \forall t \in[0, T], x, y \in \mathcal{C}_{T}^{d}.\end{equation}

This condition is rather general, under which quantitative propagation of chaos and even the existence and uniqueness of the McKean-Vlasov SDE \eqref{mckeanvariant} are far from obvious. A major breakthrough in this direction is Theorem 2.10 of Lacker \cite{lacker2021hierarchies}, where he proposed the following (slightly more restrictive) conditions on $b_0$ and $b$:
for some $0<K<\infty$,
\begin{equation}\label{1.2}
\left.\begin{array}{l}
\left|b_{0}(t, x)\right|+|b(t, x, y)| \leq K\left(1+\|x\|_{t}+\|y\|_{t}\right) \\
\left|b(t, x, y)-b\left(t, x, y^{\prime}\right)\right| \leq K\left(1+\|y\|_{t}+\left\|y^{\prime}\right\|_{t}\right)
 .\end{array} \quad \forall t \in[0, T], x, y, y^{\prime} \in \mathcal{C}_{T}^{d}\right\}.
\end{equation} 

Under \eqref{1.2}, the transport type inequality \eqref{maintransport} can be verified and we obtain the quantitative propagation of chaos result with rate $O(k^2/n^2)$ in relative entropy. Weak existence and uniqueness of McKean-Vlasov SDE under \eqref{1.2} are also established in \cite{lacker2021hierarchies}.

We will prove in this paper that the linear growth condition \eqref{linear} alone is sufficient for quantitative propagation of chaos, but on a fixed time interval $[0,T^*]$, with $T^*$ depending on various parameters of the system. In the companion paper \cite{han2022solving}, weak well-posedness of the McKean-Vlasov SDE \eqref{mckeanvariant} under \eqref{linear} have been established, extending Proposition 7.2 of \cite{lacker2021hierarchies}. However, the transport type inequality \eqref{maintransport} cannot be verified, so we cannot set up the corresponding BBGKY hierarchy and obtain the $O(k^2/n^2)$ rate.

\begin{theorem}
\label{theorem01} Consider the interacting particle system 
\begin{equation}\label{variant1}dX_t^i=\left(b_0(t,X^i)+\frac{1}{n-1}\sum_{j=1,j\neq i}^n b\left(t,X^i,X^j\right)\right)dt+dW^i_t,\quad i=1,\cdots,n,\end{equation}
where $W^{1},\cdots,W^{n}$ are $n$ independent $d$-dimensional Brownian motions. 
Denote by $P^{(n,n)}\in\mathcal{P}(\left(\mathcal{C}_T^d\right)^n)$ the joint law of $(X^1,\cdots,X^n)|_{[0,T]}$. 

Consider also the limiting McKean-Vlasov SDE, whose well-posedness follows from \cite{han2022solving}:
\begin{equation}\label{mckeanvariant}dX_t=\left(b_0(t,X)+\langle b(t,X,\cdot),\mu\rangle\right)dt+dW_t,\quad\operatorname{Law}(X)=\mu,\quad\operatorname{Law}(X_0)=\mu_0.\end{equation}

Assume the following conditions:
\begin{enumerate}
\item The pair $\left(b_{0}, b\right)$ satisfies \eqref{linear}.
 \item The law $\mu_{0} \in$ $\mathcal{P}\left(\mathbb{R}^{d}\right)$ satisfies $\int_{\mathbb{R}^{d}} e^{c_{0}|x|^{2}} \mu_{0}(d x)<\infty$ for some $c_{0}>0$. \item the initial distribution $P_0^{(n,n)}\in\mathcal{P}((\mathbb{R}^d))^n$ of the particle system factorizes:  $$P_0^{(n,n)}=\mu_0^{\otimes n}.$$
 \end{enumerate} 
 
 Then, since the initial law of the particle system is assumed to be exchangeable, by weak uniqueness of the particle system, the joint law on path space $P^{(n,n)}$ is exchangeable. Denote by $P^{(n,k)}$ its $k$-th marginal, we have:
 
 \begin{enumerate}
\item Propagation of chaos on short time. For any fixed constant $M<\infty$, we can find some $0<T_M<\infty$ depending only on $b_0$, $M$, $b$ and $\mu_0$ (thus independent of $k$ and $n$) such that 
\begin{equation}\label{convergenceof}  H\left(P^{(n,k)}[t]\mid\mu^{\otimes k}[t]\right)\leq \frac{k}{n}M, \quad k=1,\cdots,n,\quad t\in[0,T_M].
\end{equation}Via Pinsker's inequality, this implies 
$$    \|P^{(n,k)}[t]-\mu^{\otimes k}[t]\|_{\text{TV}}\leq\sqrt{\frac{2k}{n}M}<\infty,\quad k=1,\cdots,n, \quad t\in[0,T_M].$$

\item Stability of the McKean-Vlasov SDE. Fix two constants $c>0$ and $C<\infty$. For any two probability measures $\mu^1,\mu^2\in\mathcal{P}(\mathbb{R}^d)$ satisfying $\int_{\mathbb{R}^d}e^{c|x|^2}\mu^i(dx)<C<\infty,i=1,2$, denote by $\mu^1[T]$ and $\mu^2[T]$ the laws on $\mathcal{C}_T^d$ of the McKean-Vlasov SDE with initial distribution respectively $\mu^1$ and $\mu^2$. Then there exists a constant $K$ depending only on $T$, $b_0$, $b$, $c$ and $C$ such that
$$H\left(\mu^1[T]\mid\mu^2[T]\right)\leq KH\left(\mu^1\mid\mu^2\right).$$

\end{enumerate}
\end{theorem}

Combining the conclusions in (1) and (2), using also Pinsker's inequality,
we can prove a propagation of chaos result assuming the particle system has i.i.d. initial condition, in the sense that if the (product) initial law converges in relative entropy with rate $O(1/n)$ in each factor, then $P^{(n,k)}$ converges to $\mu^{\otimes k}$ with rate $O((k/n)^{1/2})$ in total variation.

We comment on the significance and limitation of this theorem.

\begin{enumerate}
    \item Although it is difficult to find interactions of a practical interest that satisfy \eqref{linear} but not \eqref{1.2} (an example in dimension 1 is $b(t,x,y)=h(x-y)$ with $h(x)=x1_{\sin(x)>0}$),  the assumption that interactions should satisfy \eqref{1.2} is conceptually more restrictive than the assumption that they should satisfy \eqref{linear}.
    \item The estimate \eqref{convergenceof} is only valid for short time $[0,T_M]$, and it is not clear to us how to generalize it to longer time intervals.

    \item We hope to prove quantitative propagation of chaos results when one replaces the driving noise by a fractional Brownian motion $B^H$. We find the argument in Theorem \ref{theorem01} can be generalized to the fractional case neatly.
\end{enumerate}

\subsection{Generalization to the case of a fractional Brownian noise}

The proof of Theorem \ref{theorem01} can be generalized, giving rise to quantitative propagation of chaos results (on short time intervals) for interacting particle systems driven by fractional Brownian noise.

In the following two theorems, the weak well-posedness of \eqref{fractionalpart} and \eqref{fractionalpart2} follow from Nualart\cite{NUALART2002103} (stated in dimension 1, but generalizes to higher dimensions straightforwardly). Weak well-posedness of the McKean-Vlasov SDE \eqref{fractionalvlasov} and \eqref{fractionalvlasov2} (with different Hurst parameters and assumptions on the drift) are proved in \cite{galeati2021distribution}, \cite{han2022solving}.

Previous works, especially \cite{WOS:000569820100011}, \cite{cass2015evolving} and \cite{10.1214/20-AOP1465}, also cover propagation of chaos results for fractional Brownian motion driven interacting particle systems. In these works the drift is assumed to be Lipschitz continuous. These  papers also serve as important motivations for this work.

\subsubsection{Singular fractional case}

\begin{theorem}\label{theorem02}  Assume that $H\in(0,\frac{1}{2}]$,   given $b_0:[0,T]\times\mathcal{C}_T^d\to\mathbb{R}^d$ and $b:[0,T]\times\mathcal{C}_T^d\times\mathcal{C}_T^d\to\mathbb{R}^d$ two progressively measurable functions, and given the initial law $\mu_0\in\mathcal{P}(\mathbb{R}^d)$. Assume that $b_0,b,\mu_0$ satisfy the same assumptions stated in in theorem \ref{theorem01}.

Consider the interacting particle system
\begin{equation}\label{fractionalpart}dX_t^i=\left(b_0(t,X^i)+\frac{1}{n-1}\sum_{j=1,j\neq i}^n b\left(t,X^i,X^j\right)\right)dt+dB^{H,i}_t,\quad i=1,\cdots,n,\end{equation}
where $B^{H,1},\cdots,B^{H,n}$ are $n$ independent $d$-dimensional fractional Brownian motions with Hurst index $H$. Assume the initial condition $X_0^1,\cdots,X_0^n$ are i.i.d with law $\mu_0$.

Consider also the McKean-Vlasov SDE
\begin{equation}\label{fractionalvlasov}dX_t=\left(b_0(t,X)+\langle b(t,X,\cdot),\mu\rangle\right)dt+dB^H_t,\quad\operatorname{Law}(X)=\mu,\quad\operatorname{Law}(X_0)=\mu_0.\end{equation}
 By weak uniqueness of \eqref{fractionalpart}, the joint law $P^{(n,n)}[T]\in\mathcal{P}((\mathcal{C}_T^d)^n)$ of $(X^1,\cdots,X^n)_{[0,T]}$ is exchangeable. Denote by $P^{(n,k)}$ its $k$-th marginal, then we have propagation of chaos on short time: we can find constants $M<\infty$ and $0<T_M<\infty$ depending only on $b_0$, $b$, $H$ and $\mu_0$ (thus independent of $k$ and $n$) such that 

\begin{equation}\label{convergenceof3}  H\left(P^{(n,k)}[t]\mid\mu^{\otimes k}[t]\right)\leq \frac{k}{n}M, \quad k=1,\cdots,n,\quad t\in[0,T_M].
\end{equation}
Via Pinsker's inequality, this implies 
$$    \|P^{(n,k)}[t]-\mu^{\otimes k}[t]\|_{\text{TV}}\leq\sqrt{\frac{2k}{n}M}<\infty,\quad  k=1,\cdots,n,\quad  t\in[0,T_M].$$

\end{theorem}

\subsubsection{Regular fractional case} 

In this case we assume $b_0$ and $b$ are state dependent, that is, we consider $b_0:[0,T]\times\mathbb{R}^d\to\mathbb{R}^d$ and $b:[0,T]\times\mathbb{R}^d\times\mathbb{R}^d\to\mathbb{R}^d$.

\begin{theorem}\label{theorem2}
With the Hurst index $H\in(\frac{1}{2},1)$, consider the interacting particle system

\begin{equation}\label{fractionalpart2}dX_t^i=\left(b_0(t,X_t^i)+\frac{1}{n-1}\sum_{j=1,j\neq i}^n b\left(t,X_t^i,X_t^j\right)\right)dt+dB^{H,i}_t,\quad i=1,\cdots,n,\end{equation}
and the McKean-Vlasov SDE
\begin{equation}\label{fractionalvlasov2}dX_t=\left(b_0(t,X_t)+\langle b(t,X_t,\cdot),\mu\rangle\right)dt+dB^H_t,\quad\operatorname{Law}(X)=\mu,\quad\operatorname{Law}(X_0)=\mu_0.\end{equation}

Assume the following conditions are satisfied:
\begin{itemize}

\item With constants $1>\alpha>1-1/(2H)>0$ and $\beta>H-1/2>0$,
$$|b_0(t, x)-b_0(s, y)| \leqslant C\left(|x-y|^{\alpha}+|t-s|^{\beta}\right) \quad \text { for all } s, t \in[0, T], x, y \in \mathbb{R}^{d},$$
$$|b(t, x,x')-b(s, y,y')| \leqslant C\left(|x-y|^{\alpha}+|x'-y'|^{\alpha}+|t-s|^{\beta}\right) \quad  s, t \in[0, T], x, y,x',y' \in \mathbb{R}^{d}.$$

    \item  $X_0^1,\cdots,X_0^n$ are i.i.d with law $\mu_0$. 
    \item For some  $c_{0}>0$ sufficiently small,  $\int_{\mathbb{R}^{d}} e^{c_{0}|x|^{2}} \mu_{0}(d x)<\infty$. (This moment condition is not necessary if $b_0$ and $b$ are bounded.) 
    
\end{itemize}
By weak uniqueness of \eqref{fractionalpart2}, the joint law $P^{(n,n)}\in\mathcal{P}((\mathcal{C}_T^d)^n)$ of $(X^1,\cdots,X^n)_{[0,T]}$ is exchangeable. Denote by $P^{(n,k)}$ its $k$-th marginal, then we have propagation of chaos on short time: we can find constants $M<\infty$ and $0<T_M<\infty$ depending only on $b_0$, $b$, $H$ and $\mu_0$ (thus independent of $k$ and $n$) such that 
\begin{equation}\label{convergenceof4}  H\left(P^{(n,k)}[t]\mid\mu^{\otimes k}[t]\right)\leq \frac{k}{n}M, \quad k=1,\cdots,n,\quad t\in[0,T_M].
\end{equation}
Via Pinsker's inequality, this implies 
$$    \|P^{(n,k)}[t]-\mu^{\otimes k}[t]\|_{\text{TV}}\leq\sqrt{\frac{2k}{n}M}<\infty,\quad k=1,\cdots,n, \quad t\in[0,T_M].$$

\end{theorem}

\section{Unbounded Integrable Interaction}\label{section22}

This section is devoted to the proof of Theorem \ref{integrablccase}.

\subsection{Pointwise probability density estimates} 

We will make crucial use of the joint density function of \eqref{thm1nparticlesystem}, especially the fact that the density function is bounded uniformly from above and below. This is facilitated by the fact that $\nabla\cdot V=0$. (We could consider more generally $\nabla \cdot V\in L^\infty$, which we leave as a further extension).

We import the  following results from \cite{guillin2021uniform}, which is set up in the case of smooth initial data. The proof can be found in \cite{guillin2021uniform} and is thus omitted here.
\begin{proposition}[Theorem 2 of \cite{guillin2021uniform}]\label{esttrans1} Let $\mu_0\in \mathcal{C}_\lambda^\infty(\mathbb{T}^d)$ be the initial law of \eqref{thm1mckeanvlasov}. Denote by $\bar{\rho}(t,x)$ the probability density function of $X_t$ solving the McKean-Vlasov SDE \eqref{thm1mckeanvlasov}, then we have $\bar{\rho}(t,x)\in\mathcal{C}_\lambda^\infty(\mathbb{R}^+\times \mathbb{T}^d).$
\end{proposition}

Denote by $\rho_n(t,\cdot)$ the joint probability density function  of $(X_t^1,\cdots,X_t^n)$ which solves the $n$-particle system \eqref{thm1nparticlesystem}. Then $\rho_n$ solves the following Fokker-Planck PDE
\begin{equation}\label{densityrhon}\partial_t\rho_n=-\sum_{i=1}^N\nabla_{x_i}\cdot\left(\left(\frac{1}{n-1}\sum_{j=1}^n K\left(x_i-x_j\right)\right)\rho_n\right)+\sum_{i=1}^n\Delta_{x_i}\rho_n.\end{equation}

We first work under the following regularity assumption
\begin{Assumption}\label{smoothness}
There is a smooth solution $\rho_n$ to \eqref{densityrhon}.
\end{Assumption}

This assumption will be removed at the end of proof. 

Since $\nabla\cdot K=0$  and $\rho_n$ is smooth, a similar argument as in Lemma 7 of \cite{guillin2021uniform} implies that $$\rho_n(t,\cdot)\in \mathcal{C}^\infty_\lambda((\mathbb{T}^d)^n),$$ as a consequence of $\rho_n(0,\cdot)\in \mathcal{C}^\infty_\lambda((\mathbb{T}^d)^n)$. Alternatively, one may first work with the smoothed version by mollifying $K$ (and finally take the limit in Section \ref{relaxedregularity}).

\subsubsection{Marginal densities and conditional densities}
Denote by $\rho_n^k$ the law of the first $k$ marginals of $\rho_n$, defined as
$$\rho_n^k(t,x_1,\cdots,x_k):=\int_{\mathbb{T}^{(N-k)d}}\rho_n(t,x_1,\cdots,x_n)dx_{k+1}\cdots dx_n.$$
Then for each $k=1,\cdots,n-1$,
\begin{equation}\frac{1}{\lambda}\leq\rho_n^k(t,\cdot)\leq\lambda\quad\text{ a.e. on } \mathbb{T}^{kd}.\end{equation} 

Denote by $\rho_n^{x_1,\cdots,x_k}(t,x_{k+1})$ the conditional probability density function of $X_t^{k+1}$ given $(X_t^1,\cdots,X_t^k)=(x_1,\cdots,x_k)$. Then by definition of conditional densities \footnote{This quadratic scaling $\lambda^2$ is a bit unusual. It is reflected in our assumption that the initial law of particle system should "almost" be the Lebesgue measure, see Assumption (2) of Theorem \ref{integrablccase}. Our proof does not work in the general setting.}
\begin{equation}\label{densitytrans2}\frac{1}{\lambda^2}\leq \rho_n^{x_1,\cdots,x_k}(t,x_{k+1})\leq \lambda^2\quad  a.e.\quad x_{k+1}\in \mathbb{T}^d.\end{equation}
for all  $(x_1,\cdots,x_k)\in(\mathbb{T}^d)^k$ and any $k=1,\cdots,n-1$.

\subsection{Two functional inequalities}
\subsubsection{Square integrability inequalities}
Since $K\in L^2(\mathbb{T}^d)$ ( follows from Assumption \ref{assumption1} assuming $d\geq 2$, as the $d=1$ case is not relevant to us), and $\frac{1}{\lambda}\leq \rho_ n^k\leq\lambda$, there exists a universal constant $C$ depending only on $K$ and $\lambda$ such that 
\begin{equation} \label{integrablefirst}
    \int_{(\mathbb{T}^d)^k} K^2(x_i-x_j)\rho_n^k(t,x^1,\cdots,x^k)\leq C<\infty,\quad 1\leq i< j\leq k.
\end{equation}

We also have the following: (recalling $\bar{\rho}$ is the density of \eqref{thm1mckeanvlasov}),
\begin{equation}
      \int_{(\mathbb{T}^d)^k} \left(\int_{\mathbb{T}^d}K(x_i-y)\bar{\rho}(y)dy\right)^2\rho_n^k(t,x^1,\cdots,x^k)\leq C<\infty,\quad 1\leq i\leq k.
\end{equation}
It suffices to use the fact that $K\in L^1(\mathbb{T}^d)$ and that $\frac{1}{\lambda}\leq \bar{\rho}\leq \lambda.$

Combining the previous two claims, we find some $M>0$ that depends only on $\lambda$ and $K$ such that
\begin{equation}\label{squareintegrability}\mathbb{E}\left[ \left|K(X_t^i-X_t^j)-\int_{\mathbb{T}^d}K(X_t^i-y)\bar{\rho}(y)dy\right|^2\right]<M<\infty,\quad1\leq i\neq j\leq n.\end{equation}

\subsubsection{Transport type Inequalities} 
We claim there exists a universal constant $\gamma>0$ such that for any $x\in\mathbb{T}^d$ and any $(x_1,\cdots,x_k)\in(\mathbb{T}^d)^k,$
\begin{equation}\label{transporttype}
    \left|\int_{\mathbb{T}^d} K(x-y)\left(\bar{\rho}(t,y)-\rho_n^{x_1,\cdots,x_k}(t,y)\right)dy\right|^2\leq \gamma H\left(\rho_n^{x_1,\cdots,x_k}(t,\cdot)\mid \bar{\rho}(t,\cdot) \right).
\end{equation}
The claim follows from inequality \eqref{reducedpinsker} which we state here again: 
\begin{equation}
\label{reducedpinsker2}
\left(\int \varphi d|\mu-\nu|\right)^2\leq 4CH\left(\mu\mid\nu\right), \quad C=\frac{1}{6}\int\varphi^2d\mu+\frac{1}{3}\int\varphi^2d\nu,\end{equation}
the fact that $K\in L^2(\mathbb{T}^d)$, and transition density estimates in equation \eqref{densitytrans2} and Proposition \ref{esttrans1}.

Equation \eqref{transporttype} is the hardest condition to verify when one establishes a BBGKY hierarchy. When $K$ is locally bounded, or at least when $|K|^2$ is exponentially integrable, a method to check it is to use the weighted Pinsker's inequality, as in \cite{lacker2021hierarchies}. In this case we do not need any information on the conditional density $\rho_n^{x_1,\cdots,x_k}.$ For singular $K$ this method no longer works, yet we can still close the bound using \eqref{reducedpinsker2} assuming that we know more about the conditional density $\rho_n^{x_1,\cdots,x_k}$.

\subsection{The hierarchy and computations of relative entropy}

Since the interaction $K$ is state dependent (we do not consider more general path dependent interactions as in the work of Lacker\cite{lacker2021hierarchies}), the particle system and the McKean-Vlasov SDEs are Markov processes and we choose to work with the relative entropy at each fixed time $t\in[0,T]$ rather than the relative entropy on the path space $\mathcal{C}_T^d$. This technical change is indeed crucial to our forthcoming argument.

In terms of density functions, we define $H_t^k$ as follows:
$$H_t^k:=H\left(\rho_n^k(t,\cdot)\mid \bar{\rho}^{\otimes k}(\cdot)\right),$$
where $H(\cdot\mid\cdot)$ is the relative entropy on $\mathcal{P}\left((\mathbb{T}^d)^k\right).$

For each $t>0,s>0$ denote by $P^k[t,t+s]$ the law of $(X^1,\cdots,X^k)_{[t,t+s]}$ if at time $t$ the joint law $(X_t^1,\cdots,X_t^n)$ has density $\rho_n(t,\cdot)$. This definition makes sense because the particle system \eqref{thm1nparticlesystem} is a Markov process. Similarly denote by $\mu[t,t+s]$ the law of the McKean-Vlasov SDE \eqref{thm1mckeanvlasov} (which is also a Markov process) on the time interval $[t,t+s]$, whose density at time $t$ is $\bar{\rho}(t,\cdot)$.

Then by the data processing inequality of relative entropy, $$H_{t+s}^k\leq H\left(P^k[t,t+s]\mid \mu^{\otimes k}[t,t+s]\right)\quad\text{ for each } s\geq 0.$$
and equality holds if $s=0$.

Consequently \footnote{We have non-rigorously assumed that $H_t^k$ is differentiable in $t$. This technical issue can be resolved if we write the differential inequalities in its integral form, and solutions to these inequalities are unchanged.}
$$\frac{d}{dt}H_t^k\leq \frac{d}{ds}|_{s=0} H\left(P^k[t,t+s]\mid \mu^{\otimes k}[t,t+s]\right).$$

It suffices to compute the right hand side. We do the computation similar to Lemma 4.6 of \cite{lacker2021hierarchies}, and obtain: (Note in the second line we only condition on  $X_t^1,\cdots,X_t^k$ rather than $(X^1,\cdots,X^k)_{[0,T]}$),
\begin{equation}\label{aligned}\begin{aligned}\frac{d}{dt}H_t^k\leq& \frac{k}{(n-1)^2}\mathbb{E}\left[\left|\sum_{j=2}^k \left(K(X_t^1-X_t^j)-\langle \mu_t,K(X_t^1-\cdot)\right)\right|^2\right]\\&\quad +\frac{k(n-k)^2}{(n-1)^2}\mathbb{E}\left[\left|\mathbb{E}[K(X_t^1-X_t^n)\mid X_t^1,\cdots,X_t^k]-\langle \mu_t,K(X_t^1-\cdot)\right|\right] \end{aligned}.
\end{equation}

Now we plug in inequality \eqref{squareintegrability} and find
\begin{equation}
    \frac{d}{dt}H_t^k\leq \frac{k(k-1)^2}{(n-1)^2}M+k\mathbb{E}\left[\
    \left|\mathbb{E}[K(X_t^1-X_t^n)|X_t^1,\cdots,X_t^k]-\langle\mu_t,K(X_t^1-\cdot\rangle \right|^2
    \right].
\end{equation}
(Recall that we denote by $\mu$ the density of McKean-Vlasov SDE and by $\bar{\rho}$ its density), then plug in inequality \eqref{transporttype}, we have
\begin{equation}
\frac{d}{dt}H_t^k\leq \frac{k(k-1)^2}{(n-1)^2}M+\gamma k\mathbb{E}\left[ H\left(\rho_N^{X_t^1,\cdots,X_t^k}(t,\cdot)\mid\bar{\rho}(t,\cdot)\right)
\right].
\end{equation}

The chain rule of relative entropy implies
$$\mathbb{E}\left[ H\left(\rho_N^{X_t^1,\cdots,X_t^k}(t,\cdot)\mid\bar{\rho}(t,\cdot)\right)
\right]=H_t^{k+1}-H_t^k.$$

Therefore we only need to solve 
\begin{equation}\label{differentialinequality12}
    \frac{d}{dt}H_t^k\leq \frac{k(k-1)^2}{(n-1)^2}M+\gamma k(H_t^{k+1}-H_t^k).
\end{equation}

Taking $k=n$ in \eqref{aligned}, we also get 
$$H_t^n\leq H_0^n+\frac{1}{2}nMt.$$

We now have the same set of differential inequalities as in the proof of Theorem 2.2 in \cite{lacker2021hierarchies}. Following the computations in that paper, and plug in the chaotic initial condition \ref{condition4} of Theorem \ref{integrablccase}, we obtain, if $n\geq 6e^{\gamma T}$, then for $k=1,\cdots,n$,
\begin{equation}\label{finalinequality}
    H_t^k\leq 2C\frac{k^2}{n^2}+C\exp(-2n(e^{-\gamma T}-\frac{k}{n})_+^2),
\end{equation}
the constant $C:=8(C_0+(1+\gamma)MT)e^{6\gamma T}.$

\subsection{Relaxing regularity of the density function }\label{relaxedregularity}

Now we discuss how to remove the Assumption \ref{smoothness}. This is suggested in Section 3.5 of \cite{guillin2021uniform}.  we may first work with the mollified version of \eqref{densityrhon}. We find $(\xi_\epsilon)_{\epsilon\geq 0}$ a sequence of mollifiers with $\|\xi_\epsilon\|_{L^1}=1$ with support strictly contained within $[-\frac{1}{2},\frac{1}{2}]^d$. Consider the mollified interaction $K_\epsilon:=K*\xi_\epsilon$, and consider $\rho_n^\epsilon$ the unique solution of the Fokker-Planck PDE with smooth coefficients 
$$\partial_t\rho_n^\epsilon+\frac{1}{n}\sum_{i,j=1}^n K^\epsilon(x_i-x_j)\cdot\nabla_{x_i}\rho_n^\epsilon=\sum_{i=1}^n\Delta_{x_i}\rho_n^\epsilon.$$

Then using the condition $\nabla\cdot K=0$, the following is immediate:

\begin{proposition}[Lemma 7 in \cite{guillin2021uniform}]
Let $\rho_n(0)\in\mathcal{C}_\lambda^\infty(\mathbb{T}^{dn})$. Then for all $t\geq 0$ and $\epsilon>0$, $\rho_n^\epsilon(t)\in\mathcal{C}_\lambda^\infty(\mathbb{T}^{dn})$.
\end{proposition}

We also consider $\bar{\rho}^\epsilon$ the density function of the McKean-Vlasov SDE \eqref{thm1mckeanvlasov}, with the interaction $K$ replaced by $K^\epsilon$.

Denote by $$H_t^{k,\epsilon}=H\left(\rho_n^{(k,\epsilon)}(t,\cdot)\mid \bar{\rho}^{\epsilon\otimes k}(\cdot)\right),$$
where $\rho_n^{(k,\epsilon)}$ is the $k$-th marginal density of $\rho_n^\epsilon$. Then by the computations leading to equation \eqref{finalinequality}, we have 
\begin{equation}
    H_t^{k,\epsilon}\leq 2C\frac{k^2}{n^2}+C\exp(-2n(e^{-\gamma T}-\frac{k}{n})_+^2),
\end{equation}
and the estimate is uniform in $\epsilon>0$. This is because in the estimates only the $L^p$ norms of $K_\epsilon$ is involved.

Now we take a limit of $\rho_n^\epsilon$. Since $\rho_n^\epsilon$ is uniformly bounded, we find a weakly-* converging subsequence in $L^\infty(\mathbb{R}_+\times\mathbb{T}^{nd})$, such that $\rho_n^\epsilon$ converges to $\rho_n$ weakly-*. Then we can check (see the proof after Lemma 7 of \cite{guillin2021uniform}) that $\rho_n$ is a weak solution of \eqref{densityrhon}. We clearly have $\frac{1}{\lambda}\leq \rho_n\leq \lambda$ by the weak-* convergence.

The proof of the theorem concludes once we show $\lim_{\epsilon\to 0}H_t^{k,\epsilon}\leq H_t^k$, which follows immediately from the lower semi-continuity of relative entropy. This completes the proof of Theorem \ref{integrablccase}.

\section{Interaction of Linear growth}
\label{section7}
This Section is devoted to the proof of Theorem \ref{theorem01}.

\subsection{An exponential moment estimate}\label{section3.1} The estimate in this section is proved in \cite{lacker2021hierarchies}, see also \cite{han2022solving}. We recover the argument here for the reader's convenience.

Consider the McKean-Vlasov SDE
$$dX_t=\left(b_0(t,X)+\langle b\left(t,X,\cdot\right),\mu\rangle\right) dt+dW_t,\quad \operatorname{Law}(X)=\mu.$$
 Take expectations on both sides and apply the linear growth condition, for $t\in[0,T]$,
$$\mathbb{E}\|X\|_t\leq |X_0|+KT+K\int_0^t \left(\|X\|_s+\mathbb{E}\|X\|_s\right)ds+\|W\|_t.$$

We take expectation on both sides, then Gronwall's inequality yields, under the assumption $\mathbb{E}\|X\|_T<\infty$ (there is a unique weak solution with $\mathbb{E}\|X\|_T<\infty$, see \cite{han2022solving}), that $$\mathbb{E}\|X\|_T\leq e^{2KT}\left(\mathbb{E}|X_0|+4dT\right).$$

Having bounded $\mathbb{E}\|X\|_T$, we apply Gronwall again and obtain 
$$\|X\|_T\leq e^{kT}\left(\|X_0|+KT+KT\mathbb{E}\|X\|_T+\|W\|_T\right).$$

Assume the initial law $\mu_0$ satisfies, for some $c_0>0$ and $C_0<\infty$, $\int_{\mathbb{R}^d}e^{c_0|x|^2}\mu_0(dx)<C_0$, then by Fernique's theorem, there exists $c>0$ and $C<\infty$ such that \begin{equation}\label{expoestimate}\mathbb{E}e^{c\|X\|_T^2}<C<\infty.\end{equation}
The constants $c$ and $C$ depend only on $T$, $b_0$, $b$, $c_0$ and $C_0$.

\subsection{A brief review of concentration inequalities}
Materials in this subsection are standard and can be found for example in  \cite{Boucheron2013ConcentrationI} and  \cite{Cunha13}.

\begin{proposition}\label{app4}
For a centered random variable $X$, the following statements are equivalent:

(1) Laplace transform condition: $\quad \exists b>0, \quad \forall t \in \mathbb{R}, \quad \mathbb{E} e^{t X} \leq e^{b^{2} t^{2} / 2}$;

(2) subgaussian tail estimate: $\exists c>0, \quad \forall \lambda>0, \quad \mathbb{P}(|X| \geq \lambda) \leq 2 e^{-c \lambda^{2}}$;

(3) $\psi_{2}$-condition: $\quad \exists a>0, \quad \mathbb{E} e^{a X^{2}} \leq 2$.
\end{proposition}

\begin{proposition}[Hoeffding's Inequality] Let $X_{1}, \ldots, X_{n} \sim X$ be i.i.d. sub-Gaussian random variables with variance proxy $b^{2}$ as given in proposition \ref{app4} (1). Then, for any $\varepsilon \geq 0$ we have
$$
\mathbf{P}\left(\frac{1}{n} \sum_{i=1}^{n} X_{i}-\mathbf{E} X \geq \varepsilon\right) \leq e^{-n \varepsilon^{2} /\left(2 b^{2}\right)}.
$$
\end{proposition}

\begin{lemma}\label{lemma6}
Let $X$ be a random variable with $E X=0$. If for some $v>0$,
$$
\boldsymbol{P}\{X>x\} \vee \boldsymbol{P}\{-X>x\} \leq e^{-x^{2} /(2 v)} \quad \text { for all } x>0,
$$
then for every integer $q \geq 1$,
$$
E\left[X^{2 q}\right] \leq 2 q !(2 v)^{q} \leq q !(4 v)^{q}.
$$
\end{lemma}
\subsection{Uniform upper bound via exponential concentration} 

Notations in this section are very different from the rest of the article. We use these notations to be consistent with \cite{jabir2019rate}, as we will use the main theorem of this paper.

On a probability space $\left(\Omega, \mathcal{F},\left(\mathcal{F}_{t} ; 0 \leq t \leq T\right), \mathbb{P}\right)$ define $n$ independent copies of the McKean-Vlasov SDE with the same initial law $\mu_0$, for $i=1,\cdots,n$:
$$d X_{t}^{i,\infty}=\left(b_{0}(t, X^{i,\infty})+\langle\mu, b(t, X^{i,\infty}, \cdot)\rangle\right) d t+d W_{t}, \quad \mu=\operatorname{Law}(X^{i,\infty}).$$
Define $$\triangle b_{t}^{i,n, \infty}=\left\langle b\left(t,X^{i, \infty},\cdot\right), \bar{\mu}^{n,i, \infty}\right\rangle-\left\langle b\left(t,X^{i, \infty},\cdot \right), \mu\right\rangle$$
where $\bar{\mu}^{n,i, \infty}=\frac{1}{n-1} \sum_{j=1,j\neq i}^{n} \delta_{\left\{X^{j, \infty}\right\}}$.
\begin{lemma}\label{lemma7}
On the time interval $[0,T]$ there exists some $0<\beta<\infty$ such that for any $0<T_0<T<\infty$, any $\delta>0$ and any integer $p\geq 1$, 

\begin{equation}\mathbb{E}_{\mathbb{P}}\left[\left(\int_{T_{0}}^{\left(T_{0}+\delta\right) \wedge T}\left|\triangle b_{t}^{i, n, \infty}\right|^{2} d t\right)^{p}\right] \leq \frac{p ! \beta^{p} \delta^{p}}{n^{p}},\quad i=1,\cdots,n.\end{equation}
\end{lemma}

\begin{proof}

The random variables $$b(t,X^{1,\infty},X^{j,\infty})-\left\langle b\left(t,X^{1, \infty},\cdot \right), \mu\right\rangle,\quad j=2,\cdots,n,$$ are centered i.i.d. random variables under $\mathbb{P}$. Moreover they have squared exponential moments: for some $c>0$,
$$\mathbb{E}_\mathbb{P}e^{c\left\|b(t,X^{1,\infty},X^{j,\infty})-\left\langle b\left(t,X^{1, \infty},\cdot \right), \mu\right\rangle\right\|^2}<C<\infty,$$
thanks to the linear growth condition \eqref{linear} and the exponential moment estimate \eqref{expoestimate}. Therefore by dominated convergence we may choose $a>0$ sufficiently small that $\psi_2$-condition in proposition \ref{app4} is satisfied. We now apply Höeffding's inequality, and then apply Lemma \ref{lemma6} ( with $\frac{1}{n}\sum_{i=1}^n X_i$ in place of $X$ and $\beta^2/n$ in place of $v$) to obtain: there exists some $\beta>0$ depending only on $T$, $b_0$, $b$ and $\mu_0$ such that
\begin{equation}
\mathbb{E}_{\mathbb{P}}\left[\left|\sum_{j=2}^{n}\left(b(t,X^{1,\infty},X^{j,\infty})-\left\langle b\left(t,X^{1, \infty},\cdot \right), \mu\right\rangle\right)\right|^{2 p}\right] \leq p !\left((n-1)\beta \right)^{p}.
\end{equation}

Setting
$$
\triangle b_{t}^{i, j, n}:=b(t,X^{i,\infty},X^{j,\infty})-\left\langle b\left(t,X^{i, \infty},\cdot \right), \mu\right\rangle,
$$
Jensen's inequality yields
$$
\begin{aligned}
&\mathbb{E}_{\mathbb{P}}\left[\left(\int_{T_{0}}^{T_{0}+\delta}\left|\frac{1}{n-1} \sum_{j=1,j\neq i}^{n} \triangle b_{t}^{i, j, n}\right|^{2} d t\right)^{p}\right]\\& \leq \frac{\delta^{p-1}}{(n-1)^{2 p}} \int_{T_{0}}^{T_{0}+\delta} \mathbb{E}_{\mathbb{P}}\left[\left|\sum_{j=1,j\neq i}^{N} \triangle b_{t}^{i, j, n}\right|^{2 p}\right] d t \\
&\leq \frac{\delta^{p} p !\beta^p}{(n-1)^{p}}\leq \frac{\delta^{p} p !(2\beta)^p}{n^{p}}.
\end{aligned}
$$
The proof completes with replacing $\beta$ by $2\beta$.
\end{proof}

From the previous proposition, we can get the following (uniform in $n$) bound. The proof of the next proposition can be found as in Proposition 3.1 of \cite{jabir2019rate}, and we will give a sketch of proof in a more complicated case after the statement of Proposition \ref{proposition8}.

\begin{proposition}\label{proposition80}
For all $0\leq T_0<T_0+\delta\leq T$, $0<\kappa<\infty$, 
\begin{equation}
\sup _{n} \mathbb{E}_{\mathbb{P}}\left[\exp \left\{\kappa \sum_{i=1}^{n} \int_{T_{0}}^{T_{0}+\delta} \triangle b_{t}^{i, n,\infty} \cdot d W_{t}^{i}-\frac{\kappa}{2}\sum_{i=1}^n \int_{T_{0}}^{T_{0}+\delta}\left|\triangle b_{t}^{i, n,\infty}\right|^{2} d t\right\}\right]
\end{equation}
is bounded from above by  $e^{\kappa}+2C+C'$ provided that   $\delta<(16 \max\{\kappa^2,1\} \beta)^{-1}$. Here $\beta$ is given in lemma \ref{lemma7}, $C$ and $C'$ two universal constants.
\end{proposition}

Denote by $$Z_t^n:=\exp\left\{\sum_{i=1}^n\int_0^t \triangle b_s^{i,n,\infty}\cdot dW_s^i-\frac{1}{2}\sum_{i=1}^n \int_0^t \left|\triangle b_s^{i,n,\infty}\right|^2ds\right\},$$

Then via Cauchy-Schwartz inequality, Proposition \ref{proposition80} implies that

\begin{equation}
\label{exponentialboundsuni2016}
\mathbb{E}_\mathbb{P}\left[(Z_t^n)^\kappa\right]<M_\kappa<\infty\quad \text{ for all }n\geq 1, \kappa>0, t\in [0,T_\kappa],\end{equation} with the upper bound $M_\kappa$ and the terminal time $T_\kappa$ depending only on ($\kappa$, $b_0$, $b$, the initial law $\mu_0$), and is uniform in the number of particles $n$. \footnote{\label{footnote123}At this point, one may be tempted to apply Theorem 2.1 of \cite{jabir2019rate}, which gives $O((k/n)^{1/2})$ rate in total variation on any finite time interval $[0,T]$. However, though passing from equation (3.3) to (3.4) in that paper is correct, passing from (3.8) to (3.9) is not because particles labeled $1$ to $k$ are not independent from the Brownian motions labeled $k+1$ to $N$, as there is an interaction between them after time $m\delta$. Thus we only get propagation of chaos results on $[0,T^*]$ with $T^*$ predetermined by parameters of the system.}

By definition of relative entropy, we have
\begin{equation}
    H\left(P^{(n,n)}[t]\mid \mu^{\otimes n}[t]\right)=\mathbb{E}_\mathbb{P}[Z_t^n\log Z_t^n].
\end{equation}
Using inequality \eqref{exponentialboundsuni2016}, that $Z_t^n$ is non-negative, and the elementary inequality $x\log x\leq x^\kappa$ for some $\kappa>0$,
we may find some upper bound $M<\infty$ depending only on $b_0,b,\mu_0$, as well as a finite time horizon $[0,T_M]$ with $T_M$ depending on the same set of parameters, such that
\begin{equation}
    H\left(P^{(n,n)}[t]\mid \mu^{\otimes n}[t]\right)<M<\infty\text{ uniformly in }n, \quad t\in[0,T_M].
\end{equation}
By the subadditivity property of relative entropy with respect to product measure (see for example \cite{hauray2014kac}, Lemma 3.3-iv), and exchangeability of the $n$-particle system, we have
\begin{equation} \label{conclusion12016}
     H\left(P^{(n,k)}[t]\mid \mu^{\otimes k}[t]\right)\leq\frac{k}{n} H\left(P^{(n,n)}[t]\mid \mu^{\otimes n}[t]\right)\leq \frac{k}{n}M<\infty,\quad t\in[0,T_M].
\end{equation}

By Pinsker's inequality, this implies
\begin{equation}  \label{conclusion22016}
    \|P^{(n,k)}[t]-\mu^{\otimes k}[t]\|_{\text{TV}}\leq\sqrt{\frac{2k}{n}M}<\infty,\quad t\in[0,T_M].
\end{equation}

\subsection{Weighted Pinsker's inequality}
For probability measures $\nu$ and $\nu'$ on the same measure space, given $f$ a measurable $\mathbb{R}^d$-valued function defined on it, we have the weighted Pinsker's inequality (see \cite{10.1007/BFb0085169} and equation (6.1) of \cite{lacker2021hierarchies}),
\begin{equation}\label{weightedpinsker}
|\langle \nu-\nu',f\rangle|^2\leq 2\left(1+\log \int e^{|f|^2}d\nu'\right)H\left(\nu\mid\nu'\right).\end{equation}

\subsection{Stability analysis of McKean-Vlasov SDEs} 

Fix some $c>0$ and $0<C<\infty$. Consider $\mathcal{P}_{c_0,C}$, the subset of $\mathcal{P}(\mathbb{R}^d)$ consisting of all probability measures $\mu$ such that 
$$\int_{\mathbb{R}^d}e^{c_0|x|^2}\mu(dx)<C<\infty.$$

Fix two measures $\mu^1$ and $\mu^2$ in $\mathcal{P}_{c_0,C}$.
Denote by $\mu^1[T]$ and $\mu^2[T]$ the laws of the McKean-Vlasov SDE on $\mathcal{C}_T^d$ with initial distribution respectively $\mu^1$ and $\mu^2$. Via Girsanov transform, we compute the relative entropy between $\mu^1[T]$ and $\mu^2[T]$:
$$H\left(\mu^1[T]\mid\mu^2[T]\right)\leq H\left(\mu^1\mid\mu^2\right)+\frac{1}{2}\int_0^T \left|\left\langle b(t,x,\cdot),\mu^1[t]-\mu^2[t]\right\rangle\right|^2 dt\mu^1[T](dx).$$

The weighted Pinsker's inequality \eqref{weightedpinsker} implies that, choosing $\epsilon>0$ sufficiently small, 
$$\left|\left\langle b(t,x,\cdot,\mu^1[t]-\mu^2[t]\right\rangle\right|^2 \mu^1[T](dx)\leq C_\epsilon H\left(\mu^1[t]\mid\mu^2[t]\right).$$
where 
$$\begin{aligned}&C_\epsilon=2\epsilon^{-1} \left(1+\log\left(\int_{\mathcal{C}_T^d}e^{\epsilon |b(t,x,y)|^2}\mu^2[t](dy)\right)\mu^1[t](dx)\right)\\&\quad\leq 2\epsilon^{-1}\int_{\mathcal{C}_T^d}e^{\epsilon |b(t,x,y)|^2}\mu^1[t](dx)\mu^2[t](dy).\end{aligned}$$
When $\epsilon>0$ is sufficiently small but fixed (we are not letting $\epsilon\searrow 0$), this expression is finite by the linear growth condition, Cauchy-Schwartz, and the exponential moment estimate \eqref{expoestimate}. Moreover, the constant $C_\epsilon$ can be bounded uniformly in the choice of $\mu^1$ and $\mu^2$.

Then, by Gronwall's lemma, we conclude there exists some constant $K>0$ depending on $T$, $b_0$, $b$, $c_0$ and $C$, such that 
$$H\left(\mu^1[T]\mid\mu^2[T]\right)\leq e^{KT}H\left(\mu^1\mid\mu^2\right).$$

We have proved the desired stability result.

\section{The case of fractional Brownian driving noise}
\label{section44}

\subsection{A note on fractional Brownian motion} We will only present the most relevant information on fractional Brownian motion $B^H$ for $H\in(0,1)$. We mention \cite{galeati2021distribution} and \cite{galeati2021noiseless} for a comprehensive discussion of fractional Brownian motion driven SDEs.

Let $B^{H}=\left\{B_{t}^{H}, t \in[0, T]\right\}$ be a $d$-dimensional fractional Brownian motion with Hurst parameter $H \in(0,1)$. That is, $B^H$ is a centered Gaussian process with covariance
$$
R_{H}(t, s)=E\left(B_{t}^{H}\otimes  B_{s}^{H}\right)=\frac{1}{2}\left\{|t|^{2 H}+|s|^{2 H}-|t-s|^{2 H}\right\} I_d \text{ for all }s,t\geq 0.
$$
The Girsanov transform of fractional Brownian motion will be used in an essential way.

\subsubsection{Girsanov Transform}

Given a fractional Brownian motion $B^H$ on a probability space, by classical results of fractional calculus, we may construct a Brownian motion $W$ on the same probability space that satisfies
$$B_t^H=\int_0^t K_H(t,s)dW_s,$$where $K_H$ is the  Volterra integral kernel given in \cite{NUALART2002103}.

Corresponding to this kernel $K_H(t,s)$ is a functional $K_H:L^2([0,T])\to I_{0^+}^ {H+\frac{1}{2}}(L^2([0,T])$, where $I_{0^+}^ {H+\frac{1}{2}}(L^2([0,T])$ is the image of $L^2([0,T]$ under the map
$$I_{0^+}^\alpha f(x)=\frac{1}{\Gamma(\alpha)}\int_0^x (x-y)^{\alpha-1}f(y)dy.$$

The precise expression of $K_H$ is given in \cite{NUALART2002103}. We will however work with its inverse $K_H^{-1}:I_{0^+}^ {H+\frac{1}{2}}(L^2([0,T]))\to L^2([0,T]).$ For $H\in(\frac{1}{2},1)$ an easy-to-use upper bound of $K_H^{-1}$ is given in \eqref{inversenormestimate}, and for $H\in(0,\frac{1}{2})$ an easy-to-use upper bound of $K_H^{-1}$ is given in \eqref{singfracbound}. 

The reason for introducing $K_H^{-1}$ is it will appear in an essential way in the Girsanov transform for fractional Brownian motions. In the particular case $H=\frac{1}{2}$, $K_H^{-1}f=f'$ for $f$ absolutely continuous, and the following Proposition reduces to the standard Girsanov transform for Brownian motion.

\begin{proposition}[ \cite{NUALART2002103}, Theorem 2]
\label{thm1}
Consider the shifted process $\tilde{B}_{t}^{H}=B_{t}^{H}+\int_{0}^{t} u_{s} \mathrm{~d} s$ defined by a process $u=\left\{u_{t}, t \in[0, T]\right\}$ with integrable trajectories. Assume that

(i) $\int_{0}^{\cdot} u_{s} \mathrm{~d} s\in I_{0^+}^ {H+\frac{1}{2}}(L^2([0,T])) $,

(ii) $E\left(\xi_{T}\right)=1$, where
$$
\xi_{T}=\exp \left(-\int_{0}^{T}\left(K_{H}^{-1} \int_{0}^{\cdot} u_{r} \mathrm{~d} r\right)(s) \mathrm{d} W_{s}-\frac{1}{2} \int_{0}^{T}\left(K_{H}^{-1} \int_{0}^{\cdot} u_{r} \mathrm{~d} r\right)^{2}(s) \mathrm{d} s\right).
$$

Then the shifted process $\tilde{B}^{H}$ is an $\mathscr{F}_{t}^{B^{H}}$-fractional Brownian motion with Hurst parameter $H$ under the new probability $\tilde{P}$ defined by $\mathrm{d} \tilde{P} / \mathrm{d} P=\xi_{T}$.
\end{proposition}

\subsection{Exponential moment estimates}

The estimates in this section are almost identical to those in \cite{lacker2021hierarchies} and in Section \ref{section3.1}. We need only replace the Brownian motion $W$ by the fractional Brownian motion $B^H$.

For a fractional Brownian motion $B^H$, we learn from Fernique's theorem that $$\mathbb{E}\|B^H\|_T\leq C_T \text{ for some } C_T>0,\quad  \mathbb{E}e^{c\|B^H\|^2_T}<\infty\text{ for some }c>0.$$
Consequently, the estimates in Section
\ref{section3.1} carry over to the McKean-Vlasov SDE
$$dX_t=\left(b_0(t,X)+\langle b\left(t,X,\cdot\right),\mu\rangle\right) dt+dB^H_t,\quad \operatorname{Law}(X)=\mu.$$
We have \begin{equation}\label{fracexpoest}\mathbb{E}\|X\|_T\leq e^{2KT}\left(\mathbb{E}|X_0|+C_T\right)
\text{ and   }\mathbb{E}e^{c\|X\|_T^2}<C<\infty \text{ for some }c>0.\end{equation}

\subsection{The regular fractional case}\label{regularcase}

The following estimate of $K_H^{-1}$ in the $H\in(\frac{1}{2},1)$ case is quite useful. See equation (6.2) of \cite{han2022solving}.

For a progressively measurable process $(u_r)_{r\geq 0}$, we have
\begin{equation}\label{inversenormestimate}
\left|K_H^{-1}\left(\int_0^\cdot u_rdr\right)(s)\right|\leq C_H\left(s^{\frac{1}{2}-H}\|u\|_{\infty;[0,s]}+s^\epsilon \|u\|_{\gamma;[0,s]}\right),
\end{equation}
where $\|u\|_{\gamma;[0,s]}$ denotes the $\gamma$-Hölder norm of $u$ on $[0,s]$ and $\|u\|_{\infty;[0,s]}$ denotes the supremum norm of $u$ on $[0,s]$.

\subsubsection{Moment estimates}
On the probability space $\left(\Omega,\mathcal{F},\mathbb{P}\right)$ define $n$ independent processes $$X^{1,\infty},\cdots,X^{n,\infty},$$ such that each $X^{i,\infty}$ solves the Mckean-Vlasov SDE 
$$
d X_{t}^{i,\infty}=\left(b_{0}\left(t,X_{t}^{i,\infty}\right)+\left\langle\mu_{t}, b\left(t,X_{t}^{i,\infty}, \cdot\right)\right\rangle\right) d t+d B_{t}^{i,H}, \quad \mu_{t}=\operatorname{Law}\left(X_{t}^{i,\infty}\right),
$$
where $B^{1,H},\cdots,B^{n,H}$ are $n$ independent $d$-dimensional fractional Brownian motions.

We introduce the notation
\begin{Definition}Define $$\triangle K   _t^{i,n}:=K_H^{-1}\left(\int_0^\cdot\frac{1}{n-1}\sum_{j=1,j\neq i}^n b\left(r,X^{i,\infty}_r,X^{j,\infty}_r\right)-\langle b\left (r,X_r^{i,\infty},\cdot\right),\mu_r\rangle dr\right)(t)$$
\end{Definition}
For $i\neq j$, we have $$\mathbb{E}_\mathbb{P}\left[K_H^{-1}(\int_0^\cdot b(r,X_r^{i,\infty},X_r^{j,\infty})dr)(t)-K_H^{-1}(\int_0^\cdot \langle b(r,X_r^{i,\infty},\cdot),\mu_r\rangle dr)(t)\right]=0.$$

Since $b$ is $\alpha>1-\frac{1}{2H}$-Hölder in its spatial variables, by triangle inequality we have, for any process $X_\cdot$ and $Y_\cdot$ in $\mathcal{C}_T^d$ (see (6.5) of \cite{han2022solving}),
$$\|b(\cdot,X_\cdot,Y_\cdot)\|_{\gamma;[0,T]}\leq_c \|X\|^\alpha_{\frac{\gamma}{\alpha};[0,T]}+\|Y\|^\alpha_{\frac{\gamma}{\alpha};[0,T]},$$ 
where $\leq_c$ means the inequality holds modulo constants depending only on $H$ and $\gamma$.

It is justified in Proposition 6.2 of \cite{han2022solving} that for some constants $K(\lambda)$ depending on $\lambda$, $b_0$, $b$ and $\mu_0$,
$$\mathbb{E}\left[\exp(\lambda\|X_\cdot^{i,\infty}\|_{\frac{\gamma}{\alpha};[0,T]}^{2\alpha})\right]<K(\lambda)<\infty,\quad  i=1,\cdots,n,\text{ for all }\lambda\in\mathbb{R}.$$
Combining the previous two formulas, we have
$$
\mathbb{E}\left[\exp\left(\lambda \left\| b(\cdot,X_{\cdot}^{i,\infty},X_{\cdot}^{j.\infty})\right\|^2_{\gamma;[0,T]}\right)\right]\leqslant C_{H}K_H(\lambda)<\infty \quad \text{for all } \lambda \in \mathbb{R},$$
where $C_H$ depends only on $T$ and $H$, and $K_H(\lambda)$ depends on $b_0$, $b$ and $\mu_0$. 
A similar reasoning gives
$$
 \mathbb{E}\left[\exp\left(\lambda   \left\| t\mapsto \left\langle \mu_t, b\left(t,X_t^{i,\infty},\cdot\right)\right\rangle\right\|^2_{\gamma;[0,T]}\right)\right]\leq C_HK_H(\lambda)<\infty \text{ for all } \lambda\in\mathbb{R}.
$$
For the supremum norm, the linear growth property of $b$ and the moment estimate of $X^{i,\infty}$ implies that for some $c>0$,
$$\mathbb{E}_\mathbb{P}e^{c\left\|b(t,X_t^{1,\infty},X_t^{j,\infty})-\left\langle b\left(t,X_t^{1, \infty},\cdot \right), \mu_t\right\rangle\right\|^2}<C<\infty.$$
If we set $$
\triangle K_t^{i,j,n}:=K_H^{-1}\left(\int_0^\cdot b(r,X_r^{i,\infty},X_r^{j,\infty})dr\right)(t)-K_H^{-1}\left(\int_0^\cdot \left\langle b(r,X_r^{i,\infty},\cdot),\mu_r\right\rangle dr\right)(t),
$$
then using \eqref{inversenormestimate} and combining all the above estimates, we conclude that we can find some $c>0$ and $C_0>0$ such that for $t\in[0,T]$, for $i,j=1,\cdots,n,i\neq j$,
$$\mathbb{E}\left[\exp\left(\frac{c}{t^{1-2H}}\left|\triangle K_t^{i,j,n}\right|^2\right)\right]\leq C_0<\infty.$$

\begin{lemma}
On the time interval $[0,T]$ there exists a $0<\beta<\infty$ such that for any $\delta>0$, any $0<T_0<T_0+\delta\leq T$,  and any integer $p\geq 1$, 
\begin{equation}\mathbb{E}_{\mathbb{P}}\left[\left(\int_{T_{0}}^{\left(T_{0}+\delta\right) \wedge T}\left|\triangle K_{t}^{i,n}\right|^{2} d t\right)^{p}\right] \leq \frac{p ! \beta^{p} \delta^{(2-2H)p}}{n^{p}}.\end{equation}
\end{lemma}

\begin{proof}

Observe that 
$$\triangle K_t^{i,n}=\frac{1}{n-1}\sum_{j\neq i}
\triangle K_t^{i,j,n}.
$$

The finite second order exponential moment we just computed allows us to apply Hoöeffding's inequality and Lemma \ref{lemma6} to deduce that there exists some $\beta>0$ depending only on $b$, $b'$ and $T$ such that
\begin{equation}
\mathbb{E}_{\mathbb{P}}
\left[\left|\frac{1}{t^{\frac{1}{2}-H}}\sum_{j=2}^n
\triangle k_t^{1,j,n}
\right|^{2p}\right]\leq p !\left((n-1)\beta \right)^{p}.
\end{equation}

A careful application of Jensen's inequality yields, noticing that $$\int_{T_0}^{T_0+\delta}t^{1-2H}dt\leq \frac{\delta^{2-2H}}{2-2H},$$ and allowing the constant $C_H$ to change in each line,
$$
\begin{aligned}
&\mathbb{E}_{\mathbb{P}}\left[\left(\int_{T_{0}}^{T_{0}+\delta}\left|\frac{1}{n-1} \sum_{j=1,j\neq i}^{n} \triangle K_{t}^{i, j, n}\right|^{2} d t\right)^{p}\right]\\& \leq C_H\frac{\delta^{(2-2H)(p-1)}}{(n-1)^{2 p}} \int_{T_{0}}^{T_{0}+\delta}t^{1-2H} \mathbb{E}_{\mathbb{P}}\left[\left|t^{H-\frac{1}{2}}\sum_{j=1,j\neq i}^{n} \triangle K_{t}^{i, j, n}\right|^{2 p}\right] d t \\
&\leq C_H \frac{\delta^{(2-2H)(p-1)} p !\beta^p}{(n-1)^{p}}\int_{T_0}^{T_0+\delta}t^{1-2H}dt\leq \frac{\delta^{(2-2H)p} p !\beta^p}{(n-1)^{p}}.
\end{aligned}
$$
 
The proof completes with replacing $\beta$ by $2\beta$.
\end{proof}

We obtain the following uniform in $n$ bound, inspired by Proposition 3.1 of \cite{jabir2019rate}:

\begin{proposition} 
\label{proposition8}
For all $0\leq T_0<T_0+\delta\leq T$, $0<\kappa<\infty$, 
\begin{equation}
\sup _{n} \mathbb{E}_{\mathbb{P}}\left[\exp \left\{\kappa \sum_{i=1}^{n} \int_{T_{0}}^{T_{0}+\delta} \Delta K_{t}^{i,n} \cdot d W_{t}^{i}-\frac{\kappa}{2}\sum_{i=1}^n \int_{T_{0}}^{T_{0}+\delta}\left|\Delta K_{t}^{i, n}\right|^{2} d t\right\}\right]
\end{equation}
is bounded from above by a constant $C(\kappa)$ provided that   $\delta<(C\kappa^2\beta)^{-\frac{1}{2-2H}}$. Here $\beta$ is given in lemma \ref{lemma7}, and $C$ is some universal constant.
\end{proposition}

The proof of this proposition is very similar to that in \cite{jabir2019rate}, so we will only give a sketch.

Step 1. Using Taylor expansion of the exponential function, it suffices to bound 
$$
\sum_{k \geq 0} \frac{\kappa^{k}}{k !} \mathbb{E}_{\mathbb{P}}\left[\left(\sum_{i=1}^{n} \int_{T_{0}}^{T_{0}+\delta} \triangle K_{t}^{i, n} \cdot d W_{t}^{i}\right)^{k}\right].
$$
Via the elementary inequality $
r^{2 p+1} \leq 1+r^{2 p+2}$, it suffices to upper bound the $k$-th power terms for even numbers $k\in 2\mathbb{N}$.

Step 2. The optimal BDG inequality (see seciton 5 of \cite{jabir2019rate}) implies 

$$
\mathbb{E}_{\mathbb{P}}\left[\left(\sum_{i=1}^{n} \int_{T_{0}}^{T_{0}+\delta} \Delta K_{t}^{i, n} \cdot d W_{t}^{i}\right)^{2 p}\right] \leq 2^{2 p}(2 p)^{p} \mathbb{E}_{\mathbb{P}}\left[\left(\sum_{i=1}^{n} \int_{T_{0}}^{T_{0}+\delta}\left|\triangle K_{t}^{i, n}\right|^{2} d t\right)^{p}\right].
$$

Now use Jensen's inequality, exchangeability of the particle system and Lemma \ref{lemma7}, we obtain an (independent of $n$) upper bound

$$\begin{aligned}
&\mathbb{E}_{\mathbb{P}}\left[\exp \left\{\kappa \sum_{i=1}^{n} \int_{T_{0}}^{T_{0}+\delta} \Delta K_{t}^{i, n} \cdot d W_{t}^{i}\right\}\right]\\&
\leq \exp \kappa+ \sum_{p \geq 0} \frac{p ! p^{p} 2^{3 p} \delta^{(2-2H)p} \beta^{p} \kappa^{2p}}{(2 p) !}\\&+\sum_{p \geq 0} \frac{(p+1) ! (p+1)^{p+1} 2^{3 p+3} \delta^{(2-2H)(p+1)} \beta^{p+1} \kappa^{2p+1}}{(2 p+1) !}.
\end{aligned}
$$

Since $\sup _{p}p ! p^{p} /(2 p) !<\infty$ and $\sup_{p}p ! p^{p} /(2 p-1) !<\infty$, we may choose $\delta$ sufficiently small (depending on $\beta$, $\kappa$ and $H$), then the infinite series is absolutely convergent and we obtain a uniform in $n$-bound.

\subsubsection{Propagation of chaos in short time}
Denote by $$Z_t^n:=\exp\left\{\sum_{i=1}^n\int_0^t \triangle K_s^{i,n}\cdot dW_s^i-\frac{1}{2}\sum_{i=1}^n \int_0^t \left|\triangle K_s^{i,n}\right|^2ds\right\},$$

Then via Cauchy-Schwartz inequality, Proposition \ref{proposition8} implies that

\begin{equation}
\label{exponentialboundsuni}
\mathbb{E}_\mathbb{P}\left[(Z_t^n)^\kappa\right]<M_\kappa<\infty\quad \text{ for all }n\geq 1, \kappa>0, t\in [0,T_\kappa],\end{equation} with the upper bound $M_\kappa$ and the terminal time $T_\kappa$ depending only on ($\kappa$, $H$, $b_0$, $b$, the initial law $\mu_0$), and is uniform in the number of particles $n$.

By definition of relative entropy, we have
\begin{equation}
    H\left(P^{(n,n)}[t]\mid \mu^{\otimes n}[t]\right)=\mathbb{E}_\mathbb{P}[Z_t^n\log Z_t^n].
\end{equation}
Using equation \eqref{exponentialboundsuni}, 
we may find some upper bound $M<\infty$ depending only on $b_0,b,\mu_0$, as well as a finite time horizon $[0,T_M]$ with $T_M$ depending on the same set of parameters, such that
\begin{equation}
    H\left(P^{(n,n)}[t]\mid \mu^{\otimes n}[t]\right)<M<\infty\text{ uniformly in }n, \quad t\in[0,T_M].
\end{equation}
By the subadditivity property of relative entropy with respect to product measure, and exchangeability of the $n$-particle system, we have
\begin{equation} \label{conclusion1}
     H\left(P^{(n,k)}[t]\mid \mu^{\otimes k}[t]\right)\leq\frac{k}{n} H\left(P^{(n,n)}[t]\mid \mu^{\otimes n}[t]\right)\leq \frac{k}{n}M<\infty, \quad t\in[0,T_M].
\end{equation}

By Pinsker's inequality, this implies
\begin{equation}  \label{conclusion2}
    \|P^{(n,k)}[t]-\mu^{\otimes k}[t]\|_{\text{TV}}\leq\sqrt{\frac{2k}{n}M}<\infty,\quad t\in[0,T_M].
\end{equation}

 This is the end of proving Theorem \ref{theorem2}.

\subsection{The singular fractional case}

Now we turn to the case where $b_0$ and $b$ have linear growth \eqref{linear} and the driving noise is fractional Brownian $B^H$ for $H\in(0,\frac{1}{2})$. We will use notations introduced in section \ref{regularcase} for simplicity.

In the case $H\in(0,\frac{1}{2})$, it is very convenient to bound $K_H^{-1}$ (see equation (4.3) of \cite{han2022solving}): there exists a costant $C_H$ depending only on $H$ and $T$ such that, for $h$ absolutely continuous,
\begin{equation}\label{singfracbound}|K_H^{-1}h(s)|\leq C_H\sup_{0\leq r\leq s}|h'(r)|,\quad \text{ for all }s\in[0,T].\end{equation}

An elementary adaptation of the proof of Lemma \ref{lemma7} gives the following

\begin{lemma}
On the time interval $[0,T]$ there exists a $0<\beta<\infty$ such that for any $\delta>0$, any $0<T_0<T_0+\delta\leq T$,  and any integer $p\geq 1$, 

\begin{equation}\mathbb{E}_{\mathbb{P}}\left[\left(\int_{T_{0}}^{\left(T_{0}+\delta\right) \wedge T}\left|\triangle K_{t}^{i,n}\right|^{2} d t\right)^{p}\right] \leq \frac{p ! \beta^{p} \delta^{p}}{n^{p}}.\end{equation}
\end{lemma}

Then we follow the same lines of proof as in Proposition \ref{proposition8} and obtain the following proposition, whose proof is straightforward and omitted:  

\begin{proposition}
For all $0\leq T_0<T_0+\delta\leq T$, $0<\kappa<\infty$, 
\begin{equation}
\sup _{n} \mathbb{E}_{\mathbb{P}}\left[\exp \left\{\kappa \sum_{i=1}^{n} \int_{T_{0}}^{T_{0}+\delta} \triangle K_{t}^{i, n} \cdot d W_{t}^{i}-\frac{\kappa}{2}\sum_{i=1}^n \int_{T_{0}}^{T_{0}+\delta}\left|\triangle K_{t}^{i, n}\right|^{2} d t\right\}\right]
\end{equation}
is bounded from above by  $e^{\kappa}+2C+C'$ provided that   $\delta<(16 \max\{\kappa^2,1\} \beta)^{-1}$. Here $\beta$ is given in lemma \ref{lemma7}, $C$ and $C'$ two universal constants.
\end{proposition}
The rest of the arguments is identical to the $H\in(\frac{1}{2},1)$ case in Section \ref{regularcase} so we omit it. We arrive at the same inequality \eqref{conclusion1} and \eqref{conclusion2}, with constants $M$, $T_M$ depending on $H$, $b_0$, $b$ and $\mu_0$.

We have now given a sketch of proof of Theorem \ref{theorem02}.

\section{Concluding remarks}\label{concluding remarks}
Earlier studies of propagation of chaos for the $n$-particle system \eqref{nparticlesystem} towards its limiting equation \eqref{mckeanvlasov} focus on interactions that are Lipschitz continuous (see for example the monograph \cite{10.1007/BFb0085169}). Examples in mathematical physics and mathematical finance bring about the necessity to work with irregular interactions. We do no justice to summarize major results in this vastly expanding field, but refer to \cite{chaintron2021propagation}
for a contemporary review.

We list a few remaining questions that have not been well settled so far.
\begin{enumerate}
\item Weaken condition \ref{condition2} in Theorem \ref{integrablccase}.

\item The technical assumption $K\in L^p(\mathbb{T}^d)$, $p>d$ is only meant to guarantee \eqref{thm1nparticlesystem} has a unique solution as a stochastic differential equation. Without this, the computations in the proof of Theorem \ref{integrablccase} are still formally true. We hope that the well-posedness of \eqref{thm1nparticlesystem} can be established, even in the case $\|K\|_{L^d(\mathbb{T}^d)}=\infty$.

\item The case of a fractional Brownian driving noise is very difficult to deal with for lack of martingale structure and Markov property. We hope to find alternative proofs of Theorem \ref{theorem02} and \ref{theorem2} that can lead to quantitative propagation of chaos results on any finite time interval $[0,T]$ rather than fixed time interval $[0,T^*]$. In particular, can we generalize Lacker's argument in \cite{lacker2021hierarchies} to the fractional Brownian case?

\item One can study propagation of chaos for diffusion coefficient $\sigma(t,X_t,\mu)$ depending on the density $\mu$, and the interaction $b$ is nonsmooth. Indeed, even the weak uniqueness of the McKean-Vlasov SDE remains open for general $b$, see for example \cite{rockner2021well} and \cite{mishura2020existence}.
\end{enumerate}

\section*{Acknowledgements}
I am grateful to my supervisor, James Norris, for stimulating my interest in interacting particle systems, providing insightful suggestions on the organization of my research paper, and improving the presentation of my writing. I would also like to thank Ioannis Kontoyiannis and Avi Mayorcas for valuable conversations on relative entropy and fractional Brownian motion. 

I would also like to thank Daniel Lacker for pointing out a technical problem in the first draft (now replaced with Theorem \ref{integrablccase} in this paper), and for pointing out some mistakes in the current version. Zimo Hao also provided valuable information on the paper \cite{jabir2019rate}.

Finally, I am very thankful to the anonymous referees who did a very careful reading and provided very instructive suggestions.

\printbibliography
\end{document}